\documentclass[11pt,reqno]{amsart}
\usepackage{amssymb,latexsym}
\usepackage{graphicx}
\usepackage{mathtools}
\usepackage{color}
\usepackage[T1]{fontenc}
\usepackage{hyperref}
\usepackage{amsmath}
\usepackage{breqn}
\usepackage[toc,page]{appendix}
\usepackage{url}    
\newcommand{\bburl}[1]{\textcolor{blue}{\url{#1}}}

\calclayout

\makeatletter
\newcommand{\monthyear}[1]{%
  \def\@monthyear{\uppercase{#1}}}
\newcommand{\volnumber}[1]{%
  \def\@volnumber{\uppercase{#1}}}
\AtBeginDocument{%
\def\ps@plain{\ps@empty
  \def\@oddfoot{\@monthyear \hfil \thepage}%
  \def\@evenfoot{\thepage \hfil \@volnumber}}
\def\ps@firstpage{\ps@plain}
\def\ps@headings{\ps@empty
  \def\@evenhead{%
    \setTrue{runhead}%
    \def\thanks{\protect\thanks@warning}%
    \uppercase{\ }\hfil}%
  \def\@oddhead{%
    \setTrue{runhead}%
    \def\thanks{\protect\thanks@warning}%
    \hfill\uppercase{Gaussian Behavior in Zeckendorf Decompositions From Lattices}}%
  \let\@mkboth\markboth
  \def\@evenfoot{%
    \thepage \hfil \@volnumber}%
  \def\@oddfoot{%
    \@monthyear \hfil \thepage}%
  }%
\footskip=25pt
\pagestyle{headings}%
}
\makeatother

\theoremstyle{plain}
\numberwithin{equation}{section}
\newtheorem{thm}{Theorem}[section]

\newtheorem{lemma}[thm]{Lemma}

\newcommand{\ncr}[2]{{#1 \choose #2}}

\newcommand{\ignore}[1]{}

\newcommand\be{\begin{eqnarray}}
\newcommand\ee{\end{eqnarray}}
\newcommand\bea{\begin{eqnarray}}
\newcommand\eea{\end{eqnarray}}
\newcommand\ben{\begin{enumerate}}
\newcommand\een{\end{enumerate}}

\newtheorem{lem}[thm]{Lemma}

\newcommand{\R}{\ensuremath{\mathbb{R}}}

\begin{document}

\monthyear{ \ }
\volnumber{ \ }
\setcounter{page}{1}

\title{Gaussian Behavior in Zeckendorf Decompositions From Lattices}

\author{Eric Chen, Robin Chen, Lucy Guo, Cindy Jiang, Steven J. Miller, Joshua M. Siktar, Peter Yu}

\address{High School Affiliated to Renmin University, Zhongguancun Road 37, Haitian District, Beijing, China}
\email{2821524834@qq.com}

\address{The Eureka Program}
\email{chenziyang20010420@126.com}

\address{The Eureka Program}
\email{530824689@qq.com}

\address{Cheltenham Ladies' College, Gloucestershire, The United Kingdom GL50 3EP}
\email{xycindyjiang@gmail.com}

\address{Department of Mathematics and Statistics, Williams College, Williamstown, MA 01267}
\email{sjm1@williams.edu}

\address{Department of
  Mathematical Sciences, Carnegie Mellon University, Pittsburgh, PA 15213}
  \email{jsiktar@andrew.cmu.edu}

\address{The Taft School, Watertown, CT 06795}
\email{peter0201yu@gmail.com}



\thanks{We thank Ray Li for helpful conversations. This project was begun when Miller was visiting Carnegie Mellon; he thanks the institution for its hospitality. This work was supported by NSF grants DMS1265673 and DMS1561945 and the Eureka Program.}

\begin{abstract}
Zeckendorf's Theorem states that any positive integer can be written uniquely as a sum of non-adjacent Fibonacci numbers. We consider higher-dimensional lattice analogues, where a legal decomposition of a number $n$ is a collection of lattice points such that each point is included at most once. Once a point is chosen, all future points must have strictly smaller coordinates, and the pairwise sum of the values of the points chosen equals $n$. We prove that the distribution of the number of summands in these lattice decompositions converges to a Gaussian distribution in any number of dimensions. As an immediate corollary we obtain a new proof for the asymptotic number of certain lattice paths.
\end{abstract}
\maketitle

\tableofcontents

\section{Introduction}
\label{sec:introduction}

Zeckendorf's Theorem states that any positive integer can be uniquely written as the sum of non-consecutive Fibonacci numbers $\{F_n\}$, defined by $F_1 = 1, F_2 = 2$, and $F_{n + 1} = F_n + F_{n - 1}$ for all $n \geq 2$ \cite{Ze}. We call this sum a number's \textbf{Zeckendorf decomposition}, and interestingly this leads to an equivalent definition of the Fibonaccis: they are the only sequence such that every positive integer can be written uniquely as a sum of non-adjacent terms. This interplay between recurrence relations and notions of legal decompositions holds for other sequences and recurrence rules as well. Below we report on some of the previous work on properties of Generalized Zeckendorf decompositions for certain sequences, and then discuss our new generalizations to two-dimensional sequences. There is now an extensive literature on the subject (see for example \cite{Bes,Bow,Br,Ca,Day,Dem,FGNPT, Fr,GTNP,Ha, Ho, HW, Ke,Mw1,Mw2,Ste1,Ste2} and the references therein).

Lekkerkerker \cite{Lek} proved that the average number of summands in the Zeckendorf decompositions of $m \in [F_n, F_{n + 1})$ is $\frac{n}{\varphi^2 + 1} + O(1) \approx .276n$ as $n \rightarrow \infty$. Later authors extended this to other sequences and higher moments (see the previous references, in particular \cite{Ben, Dem,Dor, DG, LM, LT, Mw2}), proving that given any rules for decompositions there is a unique sequence such that every number has a unique decomposition, and the average number of summands converges to a Gaussian.

To date, most of the sequences studied have been one-dimensional; many that appear to be higher dimensional (such as \cite{CFHMN2,CFHMNPX}) can be converted to one-dimensional sequences. Our goal is to investigate decompositions that are truly higher dimensional. We do so by creating a sequence arising from two-dimensional lattice paths on ordered pairs of positive integers. A legal decomposition in $d$ dimensions will be a finite collection of lattice points for which
\begin{enumerate}
\item{each point is used at most once}, and
\item{if the point $(i_1, i_2, \dots, i_d)$ is included then all subsequent points $(i_1', i_2', \dots, i_d')$ have $i'_j \ < \ i_j$ for all $j \in \{1, 2, \dots, d\}$ (i.e., \emph{all} coordinates must decrease between any two points in the decomposition).}
\end{enumerate}

We call these sequences of points on the $d$-dimensional lattice \textbf{simple jump paths}. In Section \ref{sec:futureWork} we discuss generalizations in which we allow only some of the coordinates to decrease between two consecutive points in the path; this adds combinatorial difficulties. Note that the number we assign to each lattice point depends on how we order the points (unless we are in one dimension). For example, if $d=2$ we can order the points by going along diagonal lines, or $L$-shaped paths. Explicitly, the first approach gives the ordering \be (1,1), \ \ \ (2,1),\ (1,2), \ \ \ (3,1),\ (2,2),\ (1,3), \ \ \ \dots, \ee while the second yields \be (1,1), \ \ \ (2,1),\ (2,2),\ (1,2), \ \ \ (3,1),\ (3,2),\ (3,3),\ (2,3),\ (1,3), \ \ \ \dots. \ee For the purposes of this paper, however, it does not matter which convention we adopt as our results on the distribution in the number of summands of a legal decomposition depend only on the combinatorics of the problem, and not the values assigned to each tuple. We call the labeling attached to any choice a \textbf{Simple Zeckendorf Sequence in $d$ dimensions}, and comment shortly on how this is done.
If $d = 1$ then we denote the sequence as $\{y_a\}^{\infty}_{a = 0}$ and construct it as follows.

\begin{enumerate}

\item{Set $y_1 := 1$.}

\item{Iterate through the natural numbers. If we have constructed the first $k$ terms of our sequence, the $(k+1)$\textsuperscript{th} term is the smallest integer which cannot be written as a sum of terms in the sequence, with each term used at most once.}

\end{enumerate}

Note this sequence is just powers of 2,
\begin{eqnarray}
\begin{array}{ccccccccccc}
1 & 2 & 4 & 8 & 16 & 32 & 64 & 128 & 256 & 512 & \dots,
\end{array}
\label{ZeckendorfDiagonalSequenceSimp1D}
\end{eqnarray}
and a legal decomposition of $n$ is just its binary representation.

If $d = 2$, on the other hand, as remarked above we have choices. We describe the Simple Zeckendorf Diagonal Sequence $\{y_{a, b}\}_{a, b = 0}^{\infty}$; its construction is similar in nature to the $d = 1$ case and proceeds as follows.

\begin{enumerate}
\item{Set $y_{1,1}  :=  1$.}

\item{Iterate through the natural numbers. For each such number, check if any path of numbers in our sequence with a strict leftward and downward movement between each two points sums to the number. If no such path exists, add the number to the sequence so that it is added to the shortest unfilled diagonal moving from the bottom right to the top left.}

\item{If a new diagonal must begin to accommodate a new number, set the value $y_{k, 1}$ to be that number, where $k$ is minimized so that $y_{k, 1}$ has not yet been assigned.}
\end{enumerate}

In \eqref{ZeckendorfDiagonalSequenceSimp2D} we illustrate several diagonals' worth of entries when $d = 2$, where the elements are always added in increasing order. Note that unlike the Fibonacci sequence, we immediately see that we have lost the uniqueness of decompositions (for example, $25$ has two legal decompositions: $20+5$ and $24+1$).
\begin{eqnarray}
\begin{array}{cccccccccc}280 & \cdots & \cdots & \cdots & \cdots & \cdots & \cdots & \cdots & \cdots & \cdots \\157 & 263 & \cdots & \cdots & \cdots & \cdots & \cdots & \cdots & \cdots & \cdots \\84 & 155 & 259 & \cdots & \cdots & \cdots & \cdots & \cdots & \cdots & \cdots \\50 & 82 & 139 & 230 & \cdots & \cdots & \cdots & \cdots & \cdots & \cdots \\28 & 48 & 74 & 123 & 198 & \cdots & \cdots & \cdots & \cdots & \cdots \\14 & 24 & 40 & 66 & 107 & 184 & \cdots & \cdots & \cdots & \cdots \\7 & 12 & 20 & 33 & 59 & 100 & 171 & \cdots & \cdots & \cdots \\3 & 5 & 9 & 17 & 30 & 56 & 93 & 160 & \cdots & \cdots \\1 & 2 & 4 & 8 & 16 & 29 & 54 & 90 & 154 & \cdots \end{array}
\label{ZeckendorfDiagonalSequenceSimp2D}
\end{eqnarray}

Of course, analogous procedures to the one which creates \eqref{ZeckendorfDiagonalSequenceSimp2D} exist for higher dimensions, but the intended illustration is most intuitive in two dimensions. For the same reason as in the $d = 2$ case, there are clearly multiple procedures to generate the higher-dimensional sequences, even if one fixes restrictions on how to choose the summands in as many as $d - 2$ dimensions.

Numerical explorations (see Figure \ref{fig:gaussplots}) suggest that, similarly to other sequences mentioned earlier, the distribution of the number of summands converges to a Gaussian.

\begin{figure}[h] 
\begin{center}
\scalebox{.7}{\includegraphics[width=9cm,height=5.5cm,angle=0]{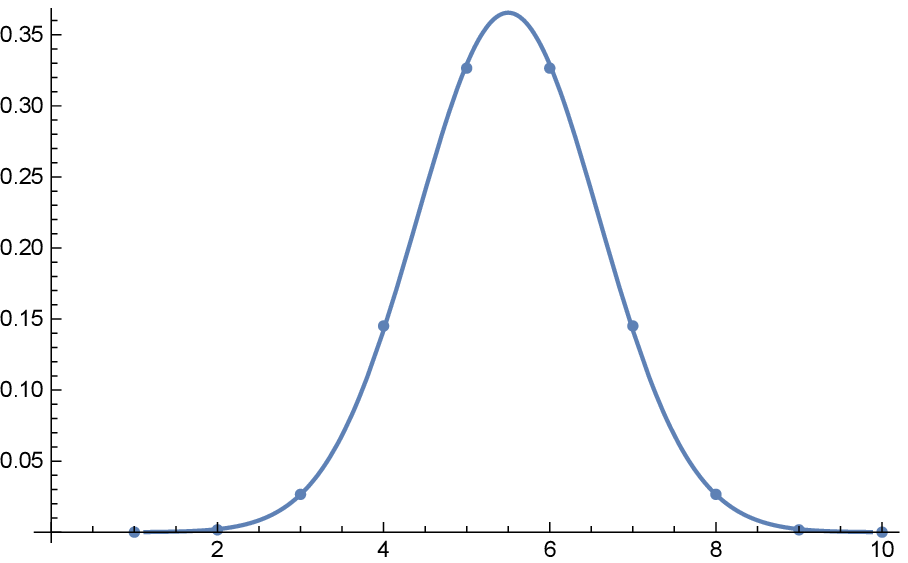}} \ \ \scalebox{.7}{\includegraphics[width=9cm,height=5.5cm,angle=0]{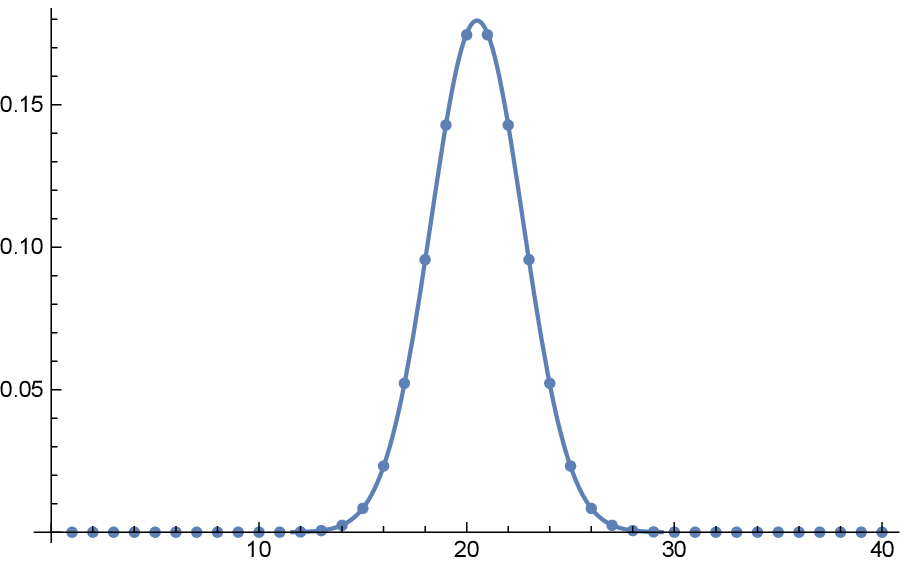}}
\caption{\label{fig:gaussplots} Distribution of the number  simple jump paths of varying lengths versus the best fit Gaussian. Left: Starting at $(10,10)$. Right: Starting at $(40, 40)$. In both cases the horizontal axis is the number of summands and the vertical axis is the probability of obtaining a simple jump path with some number of summands when selecting one from all simple jump paths uniformly at random. }
\end{center}
\end{figure}

Our main result is that as $n\to \infty$, we converge to Gaussian behavior in any number of dimensions.

\begin{thm}{($d$-dimensional Gaussianity)}
\label{ddgauss}
Let $n$ be a positive integer, and consider the distribution of the number of summands among all simple jump paths of dimension $d$ with starting point $(i, i,.....,i)$ where $ 1 \leq  i \leq n$, and each distribution represents a (not necessarily unique) decomposition of some positive number. This distribution converges to a Gaussian as $n \rightarrow  \infty$.
\end{thm}

In Section \ref{sec:simpleJumpPaths} we motivate our problem further, explore the notion of a simple jump path in more depth, and prove some needed lemmas. Then, we prove Theorem \ref{ddgauss} in Section \ref{sec:gaussianDDim}. The result is just the Central Limit Theorem for a binomial random variable if $d=1$. If $d=2$ it can be proved directly through combinatorial identities, but for larger $d$ the combinatorial lemmas do not generalize and we are forced to resort to analytic techniques. We show that the functional dependence is that of a Gaussian, and thus as the probabilities must sum to 1 the normalization constant, which depends on the number of paths, must have a certain asymptotic formula. Thus, as an immediate consequence, we obtain new proofs for the asymptotic number of paths (the approach mentioned on the OEIS uses generating functions and expansions). We end with a discussion of future work and generalizations of the simple jump paths.

\section{Properties of Simple Jump Paths}\label{sec:simpleJumpPaths}

We first set some notation for our simple jump paths. We have walks in $d$ dimensions starting at some initial point $(a_1, a_2, \dots, a_d)$ with each $a_j  > 0$, and ending at the origin $(0, 0, \dots, 0)$. Note that our simple jump paths must always have movement in all dimensions at each step. We are just adding one extra point, at the origin, and saying every path must end there. Note that as we always change all of the indices during a step, we never include a point where only some of the coordinates are zero, and thus there is no issue in adding one extra point and requiring all paths to end at the origin.

Our walks are sequences of points on the lattice grid with positive indices or the origin, and we refer to movements between two such consecutive points as \textbf{steps}. Thus a simple jump path is a walk where each step has a strict movement in all $d$ dimensions. More formally, a simple jump path of length $k$ starting at $(a_1, a_2, \dots, a_d)$ is a sequence of points $\{(x_{i, 1}, \dots, x_{i, d})\}^{k}_{i = 0}$ where the following hold:
\begin{itemize}
\item $(x_{0, 1}, \dots, x_{0, d}) \ = \ (a_1, \dots, a_d)$,

\item $(x_{k, 1}, \dots, x_{k, d}) \ = \ (0, \dots, 0)$, and

\item for each $i \in \{1, \dots, k - 1\}$ and $j \in \{1, \dots, d\}$, $x_{i, j} \ > \ x_{i + 1, j}$.

\end{itemize}

For a fixed $d$ and any choice of starting point $(n, n, \dots, n) \in \R^d$, we let $s_d(n)$ denote the number of simple jump paths from $(n, n, \dots, n)$ to the origin, and $t_d(k, n)$ the subset of these paths with exactly $k$ steps. As we must reach the origin, every path has at least 1 step, the maximum number of steps is $n$, and
\begin{equation} \label{simpleJumpPathPartitionByNumSteps}
s_d(n) \ = \ \sum_{k = 1}^{n} t_d(k, n).
\end{equation}

We now determine $t_d(k, n)$. In one dimension we have $t_d(k, n) = \ncr{n-1}{k-1}$, as we must choose exactly $k-1$ of the first $n-1$ terms (we must choose the $n$\textsuperscript{th} term as well as the origin, and thus choosing $k-1$ additional places ensures their are exactly $k$ steps). The generalization to higher dimensions is immediate as we are looking at simple paths, and thus there is movement in each dimension in each step; this is why we restrict ourselves to simple paths, as in the general case we do not have tractable formulas like the one below.

\begin{lemma}
\label{lem:enumerateSimpleJumpPaths}
For $a_1, \dots, a_d$ positive integers let $t_d(k; a_1,\dots,a_d)$ denote the number of simple paths of length $k$ starting at $(a_1, \dots, a_d)$ and ending at $(0, \dots, 0)$. Then for $1 \leq k \le \min(a_1,\dots, a_d)$,
\begin{eqnarray}
t_d(k; a_1, \dots, a_d) \ = \ncr{a_1 - 1}{k - 1}\ncr{a_2 - 1}{k - 1} \cdots \ncr{a_d - 1}{k - 1};
\end{eqnarray} if $a_1 = \cdots = a_d = n$ we write $t_d(k,n)$ for $t_d(k;a_1, \dots, a_d)$. We have
\begin{eqnarray}
s_d(n) \ = \ \sum_{k=1}^n t_d(k,n),
\end{eqnarray} and $s_1(n) = 2^{n-1}$, $s_2(n) = \ncr{2n - 2}{n - 1}$ (for higher $d$ there are no longer simple closed form expressions\footnote{We will find excellent approximations for large $n$ and fixed $d$ later.}). 
\end{lemma}

The proof is an immediate, repeated application of the one-dimensional result, with the two formulas (for $s_1(n)$ and $s_2(n)$) being well-known binomial identities (see for example \cite{Mil}).

\section{Gaussianity in $d$-Dimensional Lattices}\label{sec:gaussianDDim}

\subsection{Mean and Variance}\label{sec:meanVar}

To prove Theorem \ref{ddgauss}, we start by determining the density, $p_d(k, n)$, for the number of simple jump paths of length $k$ starting at $(n, \dots, n)$:
\begin{eqnarray}
p_d(k, n) \ := \ \frac{t_d(k, n)}{s_d(n)}. \label{dDimDensityCondensed}
\end{eqnarray}
Much, though not all, of the proof when $d = 1$ carries over to general $d$. We therefore concentrate on $d = 1$ initially and then remark on what issues arise when we generalize, and discuss the resolution of these problems.

We begin by determining the mean and standard deviation. The analysis for the mean holds for all $d$, but the combinatorial argument for the variance requires $d \le 2$. Due to the presence of $n-1$ in the formula for $t_d(k,n)$, we work with $n+1$ below to simplify some of the algebra.

\begin{lem}\label{simpleMeanStdDev} Consider all simple jump paths from $(n+1, \dots, n+1)$ to the origin in $d$-dimensions. If $K$ is the random variable denoting the number of steps in each path, then its mean $\mu_d(n + 1)$ and standard deviation $\sigma_d(n + 1)$ are \begin{eqnarray}
\mu_d(n + 1) & \ = \ & \frac{1}{2}n + 1 \label{simpleSquareLatticeMean} \end{eqnarray} and
\begin{eqnarray}\label{simpleSquareLatticeStdDev}
\sigma_1(n + 1)\ = \ \frac{\sqrt{n}}{2}, \ \ \ \ \ \sigma_2(n+1) \ = \ \frac{n}{2\sqrt{2n-1}} \ \approx \ \frac{\sqrt{n}}{2\sqrt{2}}.
\end{eqnarray} Further, we have \begin{equation} \sigma_d(n+1) \ \le \ \sigma_1(n+1) \ \le \ \sqrt{n}/2. \end{equation}
\end{lem}


\begin{proof} The results for $d=1$ are well known, as we have a binomial random variable. For $d=2$ one can compute the mean and the variance by combinatorial arguments (see Appendix \ref{sec:derivMeanStdDev}); unfortunately while these can be generalized to give the mean for any $d$ they do not generalize for the variance.

Because we must end at the origin, note each path must have length at least 1. Thus instead of studying the number of paths of length $k \in \{1, \dots, n+1\}$ we instead study the number of paths of length $\kappa \in \{0, \dots, n\}$ and then add 1 to obtain the mean (there is no need to add 1 for the variance, as the variance of $K$ and $K-1$ are the same).

As \be t_d(k;n+1) \ = \ \frac{\ncr{n}{k}^d}{s_d(n+1)}, \ee the symmetry of the binomial coefficients about $n/2$ implies the mean of $K-1$ is $n/2$. All that remains is to prove the variance bound for $d \ge 2$. Note that the variance of $K-1$ is \be \sigma_d(n+1) \ = \ \sum_{\kappa=0}^{n} \left(\kappa - n/2\right)^2 \frac{\ncr{n}{\kappa}^d}{s_d(n+1)}. \ee By symmetry it suffices to investigate $\kappa \ge n/2$. Since the binomial coefficients are strictly decreasing as we move further from the mean, for such $\kappa$ we find that \be \frac{p_d(\kappa)}{p_d(\kappa+1)} \ = \ \frac{\ncr{n}{\kappa}^d}{\ncr{n}{\kappa+1}^d} \ \ge \ 1,\ee and thus for every $g > 0$ we see that the probability of $K-1$ being within $g$ of the mean increases as $d$ increases. Thus the variance is smallest at $d=1$, completing the proof. \end{proof}



Next, we show with high probability that $K$ is close to the mean.

\begin{lem}\label{lem:chebyshevbound} Consider all simple jump paths from $(n+1, \dots, n+1)$ to the origin in $d$-dimensions. If $K$ is the random variable denoting the number of steps in each path, then the probability that $K$ is at least $n^{\epsilon} n^{1/2}/2$ from the mean is at most $n^{-2\epsilon}$. \end{lem}

\begin{proof} By Chebyshev's Inequality, \be {\rm Prob}\left(|K - (n/2 + 1)| \ \ge\ n^\epsilon \sigma_d(n+1)\right) \ \ \le\  \ \frac1{n^{2\epsilon}}. \ee As $\sigma_d(n+1) \le n^{1/2}/2$ by Lemma \ref{simpleMeanStdDev}, we only decrease the probability on the left if we replace $\sigma_d(n+1)$ with $n^{1/2}/2$, and thus the claim follows. \end{proof}

One important consequence of the above lemma is that if we write $k$ as $\mu_{d}(n + 1) + \ell n^{1/2}/2$, then with probability tending to 1 we may assume $|\ell| \le n^\epsilon$.


\subsection{Gaussianity}
\label{sec:GaussianityProof}

The proof of Theorem \ref{ddgauss} in general proceeds similarly to the $d=1$ case. For $d \le 2$ we have explicit formulas for both the variance and $s_d(n+1)$, which simplify the proof. For general $d$ we show that the resulting distribution has the same functional form as a Gaussian, and from this we obtain asymptotics for both the variance and the number of paths.

\begin{proof}[Proof of Theorem \ref{ddgauss}] From Lemma \ref{lem:chebyshevbound}, if we write
\begin{equation} \label{meanStdDevApproxOfK}
k \ = \ \mu_{d}(n + 1) + \ell n^{1/2}/2
\end{equation} then the probability of $|\ell|$ being at least $n^{1/9}$ is at most $n^{2/9}$, so in the arguments below we assume $|\ell| \le n^{1/9}$. In particular, this means that both $k$ and $n-k$ are close to $n/2$ with probability tending to 1 as $n\to\infty$. We are using $n^{1/2}/2$ and not $\sigma_d(n+1)$ as this way a quantity below will perfectly match the $d=1$ case.

For $m$ large, Stirling's Formula states that
\begin{equation} \label{stirlingCite}
m! \ = \ m^m e^{-m} \sqrt{2\pi m} \left(1 + O\left(\frac1{m}\right)\right).
\end{equation}
Thus\begin{eqnarray}\label{equation:StirlingExpansion}
p_d(k, n+1) & \ = \ & \frac{\ncr{n}{k}^d}{s_d(n+1)} \ = \ \frac1{s_d(n+1)} \left(\frac{n!}{k!(n-k)!}\right)^d \nonumber\\
& \ = \ & \frac1{s_d(n+1)}\left(\frac{\sqrt{2\pi n} n^n}{\sqrt{4\pi^2 k (n-k)} k^k (n-k)^{n-k}} \cdot \frac{\left(1 + O\left(\frac{1}{n}\right)\right)}{\left(1+O\left(\frac{1}{n-k}\right)\right)\left(1+O\left(\frac{1}{k}\right)\right)}\right)^d,\nonumber\\
\end{eqnarray} and the ratio of the big-Oh terms is $1 + O(1/n)$ since $k$ and $n-k$ are approximately $n/2$ (note the big-Oh constant here is allowed to depend on $d$, which is fixed).

We now turn to the other part of the above expression. If we divide the rest of the quantity in parentheses by $2^n$ then we have the probability in 1-dimension, whose analysis is well-known; thus \be p_d(k, n+1) \ = \ \frac{2^{nd} n^{d/2}}{s_d(n+1)} \left(\frac{n^n}{2^n k^k (n-k)^{n-k} \sqrt{2\pi k(n-k)}} \right)^d \cdot \left(1 + O(1/n)\right). \ee The quantity to the $d$-th power converges (up to the normalization factor) to a Gaussian by the Central Limit Theorem for a binomial random variable; for completeness we sketch the proof.

Using $n, n-k$ are close to $n/2$, we find
\begin{eqnarray} \label{1DGaussMainTermExpandA}
p_{{\rm main}, 1}(k) & \ := \ & \frac{n^n}{2^n k^k (n-k)^{n-k} \sqrt{2\pi k(n-k)}} \nonumber\\
 &= & \frac{1}{\sqrt{\frac{1}{2}\pi n^2}} \cdot \frac{1}{\left(1-\frac{\frac{\ell\sqrt n}{2} }{n/2}\right)^{n/2-\frac{\ell\sqrt n}{2}+\frac12}\left(1+\frac{\frac{\ell\sqrt n}{2}}{n/2}\right)^{n/2+\frac{\ell\sqrt n}{2}+\frac12}}.
\end{eqnarray}

Let $q_n$ be the denominator of the second fraction above. We approximate $\log(q_n)$ and then exponentiate to estimate $q_n$. As $|\ell| \le n^{1/9}$, when we take the logarithms of the terms in $q_n$ only the first two terms in the Taylor expansion of $\log(1+u)$ contribute as $n\to\infty$. Thus
\begin{eqnarray} \label{gaussFracExpanC}
\log q_n
& \ = \ & \left(\frac{n}{2} - \frac{\ell\sqrt n}{2} + \frac{1}{2}\right)\left(-\frac{\ell}{\sqrt n} - \frac{\ell^2}{2n}+ O\left(\frac{\ell^3}{n^{3/2}}\right)\right) \nonumber\\ & & \ \ \ \ + \ \left(\frac{n}{2} + \frac{\ell\sqrt n}{2} + \frac{1}{2}\right)\left(\frac{\ell}{\sqrt n} - \frac{\ell^2}{2n}+ O\left(\frac{\ell^3}{n^{3/2}}\right)\right)\nonumber\\
 &= & \frac{\ell^2}{2}+O\left(n\cdot \frac{n^{1/3}}{n^{3/2}} - \frac{\ell^2}{2n}\right) \ = \ \frac{\ell^2}{2} + O\left(n^{-1/6}\right),
\end{eqnarray} which implies (since $k = \mu_d(n+1) + \ell \sqrt{n}/2$)
\begin{equation}\label{gaussFracExpanD}
q_m \ = \ e^{\frac{(k-\mu_d(n+1))^2}{n/2}} e^{O(n^{-1/6})}.
\end{equation}
Thus collecting our expansions yields, for $|\ell| \le n^{1/9}$,
\begin{equation}
p_d(k, n+1) \ = \ \frac{2^{nd} n^{d/2}}{s_d(n+1)(\pi n^2 /2)^{d/2}}\ e^{-\frac{d(k-\mu_d(n+1))^2}{n/2}} \cdot e^{O(n^{-1/6})}.
\end{equation} Note the second exponential is negligible as $n\to\infty$, and the first exponential is that of a Gaussian with mean $\mu_d(n+1)$ and variance $\sigma_d(n+1)^2 = n/4d$. As this is a probability distribution it must sum to 1 (the terms with $|\ell|$ large contribute negligibly in the limit), and thus $2^{nd} / (s_d(n+1) (\pi n /2)^{d/2})$ must converge to the normalization constant of this Gaussian, which is $1/\sqrt{2\pi s_d(n+1)^2}$. In particular, we obtain\footnote{One can check this asymptotic by computing $s_d(n+1)$ for various $d$ and looking up the resulting sequences on the OEIS, which agree; for example, see the entry A182421 for the sequence when $d=7$.} \begin{equation} s_d(n+1) \ \sim \ \frac{2^{nd}  n^{d/2}}{(\pi n^2 /2)^{d/2}} \cdot \sqrt{2 \pi n/4d} \ = \  2^{nd} \left(\frac{\pi n}{2}\right)^{-\frac{d}{2}+\frac12} d^{-1/2}. \end{equation}
\end{proof}

\section{Future Work and Concluding Remarks}\label{sec:futureWork}

We could also consider the \textbf{Compound Zeckendorf Diagonal Sequence in $d$ dimensions}, which is constructed in a similar way to \eqref{ZeckendorfDiagonalSequenceSimp1D} and \eqref{ZeckendorfDiagonalSequenceSimp2D}, but allows more paths to be legal (explicitly, each step is no longer required to move in all of the dimensions). While the $d \ = \ 1$ Compound Zeckendorf Diagonal Sequence is the same as the simple one, the two notions of paths give rise to different sequences when $d \ = \ 2$. In that case, the Compound Zeckendorf Diagonal Sequence is denoted $\{z_{a, b}\}^{\infty}_{a \ = \ 0, b \ = \ 0}$, and is constructed as follows.
\begin{enumerate}
\item{Set $z_{1,1} \ := \ 1$.}

\item{Iterate through the natural numbers. For each such number, check if any path of distinct numbers without upward or rightward movements sums to the number. If no such path exists, add the number to the sequence so that it is added to the shortest unfilled diagonal moving from the bottom right to the top left.}

\item{If a new diagonal must begin to accommodate a new number, set the value $z_{k, 1}$ to be that number, where $k$ is minimized so that $z_{k, 1}$ has not yet been assigned.}
\end{enumerate}

The difference between this and the Simple Zeckendorf Diagonal Sequence is that we now allow movement in just one direction. This greatly complicates the combinatorial analysis because now the simultaneous movements in different dimensions depend on each other. In particular, if a step contains a movement in one direction, it no longer needs to contain a movement in other directions to be regarded as a legal step. In \eqref{ZeckendorfDiagonalSequenceComp} we illustrate several diagonals' worth of entries, where the elements are always added in increasing order.
\begin{eqnarray}
\begin{array}{cccccccccc}6992 & \cdots & \cdots & \cdots & \cdots & \cdots & \cdots & \cdots & \cdots \\2200 & 6054 & \cdots & \cdots & \cdots & \cdots & \cdots & \cdots & \cdots & \cdots \\954 & 2182 & 5328 & \cdots & \cdots & \cdots & \cdots & \cdots & \cdots & \cdots \\364 & 908 & 2008 & 5100 & \cdots & \cdots & \cdots & \cdots & \cdots & \cdots \\138 & 342 & 862 & 1522 & 4966 & \cdots & \cdots & \cdots & \cdots & \cdots \\44 & 112 & 296 & 520 & 1146 & 2952 & \cdots & \cdots & \cdots & \cdots \\16 & 38 & 94 & 184 & 476 & 1102 & 2630 & \cdots & \cdots & \cdots \\4 & 10 & 22 & 56 & 168 & 370 & 1052 & 2592 & \cdots & \cdots \\1 & 2 & 6 & 18 & 46 & 140 & 366 & 1042 & 2270 & \cdots\end{array} \label{ZeckendorfDiagonalSequenceComp}
\end{eqnarray}

Just as in \eqref{ZeckendorfDiagonalSequenceSimp2D}, uniqueness of decompositions does not hold in the compound case. For instance, $112 + 38 + 10$ and $140 + 18 + 2$ are both legal decompositions of $160$ in \eqref{ZeckendorfDiagonalSequenceComp}. Moreover, just like the Simple Zeckendorf Diagonal Sequences \eqref{ZeckendorfDiagonalSequenceSimp1D} and \eqref{ZeckendorfDiagonalSequenceSimp2D}, Compound Zeckendorf Diagonal Sequences can be built in higher dimensions with multiple ways of formulating how to add terms to the sequence.

Many of the articles in the literature use combinatorial methods and manipulations of binomial coefficients to obtain similar results (see, for instance, \cite{Eg, Len, Mw2}). Thus a question worth future study is to extend the combinatorial variance calculation to $d$ dimensions (see Lemma \ref{2DimStdDev}).

Finally, similar to \cite{Bow,Ko} and related work, we can investigate the distribution of gaps between summands in legal paths. One can readily obtain explicit combinatorial formulas for the probability of a given gap; the question is whether or not nice limits exist in this case as they do for the one-dimensional recurrences previously studied.



\appendix
\section{Derivation of Mean and Standard Deviation for Simple Jump Paths}\label{sec:derivMeanStdDev}
\begin{lemma}[Mean for Simple Jump Path Distribution]
If $\mu_d(i)$ denotes the mean number of steps in a $d$-dimensional simple jump path from $(i, i, \dots, i)$ to the origin, then
\begin{eqnarray}
\mu_d(n + 1) \ = \ \frac{1}{2}n + 1.\label{simpledDimLatticeMean}
\end{eqnarray}
\end{lemma}

\begin{proof}
By the definition of the first moment,
\begin{eqnarray}
\mu_d(n + 1) & \ = \ & \frac{\sum^{n + 1}_{k = 1} k \cdot t_d(k, n+1)}{s_d(n+1)}  \nonumber\\
& \ = \ & \frac{\sum^{n}_{k = 0}(k + 1)t_d(k+1, n+1)}{s_d(n+1)} \ = \  \frac{\sum^{n}_{k = 0}k\ncr{n}{k}^d + s_d(n+1)}{s_d(n+1)}.\label{firstMomentExpansion}
\end{eqnarray}

We complete the proof based on the parity of $n$. We first assume $n$ is odd. Then
\begin{eqnarray}
\sum^{n}_{k = 0}k\ncr{n}{k}^d \ = \ \sum^{\lfloor \frac{n}{2}\rfloor}_{k = 0}\left[k\ncr{n}{k}^d + (n - k)\ncr{n}{n - k}^d\right] \ = \ n\sum^{\lfloor \frac{n}{2}\rfloor}_{k = 0}\ncr{n}{k}^d.
\label{firstOrderSquareCoeffExpA}
\end{eqnarray}
Notice that by the symmetry of binomial coefficients,
\begin{eqnarray}
\sum^{\lfloor \frac{n}{2} \rfloor}_{k = 0}\ncr{n}{k}^d \ = \ \sum^{n}_{k = \lceil \frac{n}{2}\rceil}\ncr{n}{k}^d,
\label{binomCoeffSymNOdd}
\end{eqnarray}
so
\begin{eqnarray}
n\sum^{\lfloor \frac{n}{2}\rfloor}_{k = 0}\ncr{n}{k}^d \ = \ \frac{1}{2}n\sum^{n}_{k = 0}\ncr{n}{k}^d \ = \ \frac{1}{2}n s_d(n+1),
\label{firstOrderSquareCoeffExpB}
\end{eqnarray} and substituting into \eqref{firstMomentExpansion} completes the proof in this case.

Now we consider $n$ even. A similar analysis as in the previous case works, except we need to deal with the term where $k \ = \ n/2$, which is matched with itself:
\begin{eqnarray}
\sum^{n}_{k = 0}k\ncr{n}{k}^d  & \ = \  &\frac{n}{2}\ncr{n}{n/2}^d + \sum^{\frac{n}{2} - 1}_{k = 0}k\ncr{n}{k}^d + \sum^{n}_{k = \frac{n}{2} + 1}k\ncr{n}{k}^d  \nonumber\\
& \ = \ & \frac{n}{2}\ncr{n}{n/2}^d + \sum^{\frac{n}{2} - 1}_{k = 0}\left[k\ncr{n}{k}^d + (n - k)\ncr{n}{n - k}^d\right] \nonumber\\
& \ = \ & \frac{n}{2}\ncr{n}{n/2}^d + n\sum^{\frac{n}{2} - 1}_{k = 0}\ncr{n}{k}^d
\label{firstOrderSquareCoeffExpC}
\end{eqnarray}
Again utilizing the symmetry of binomial coefficients,
\begin{eqnarray}
\sum^{\frac{n}{2} - 1}_{k = 0}\ncr{n}{k}^d \ = \ \sum^{n}_{k = \frac{n}{2} + 1}\ncr{n}{k}^d,
\label{binomCoeffSymNEven}
\end{eqnarray}
so \eqref{firstOrderSquareCoeffExpC} is equivalent to
\begin{eqnarray}
\frac{n}{2}\ncr{n}{n/2}^d + \frac{n}{2}\sum_{k \in \{0, 1, \dots, n\}\setminus\{n/2\}}\ncr{n}{k}^d \ = \ \frac{n}{2}\sum^{n}_{k = 0}\ncr{n}{k}^d \ = \ \frac{n}{2} s_d(n+1),
\label{firstOrderSquareCoeffExpD}
\end{eqnarray}
completing the proof.
\end{proof}

\begin{lemma}[Standard Deviation for 2-Dimensional Simple Jump Paths]
\label{2DimStdDev}
If $\sigma_2(i)$ represents the standard deviation for the number of steps in a simple jump path in $d$-dimensions from $(i,i)$ to the origin, then
\begin{eqnarray}
\sigma_2(n+1) \ = \ \frac{n}{2\sqrt{2n - 1}}.
\label{appsimpleSquareLatticeStdDev}
\end{eqnarray}
\end{lemma}

As the variance in the one-dimensional case is well known (it is the variance of a binomial random variable), we provide details only for $d \ = \ 2$. As remarked earlier, the combinatorial approach taken below does not generalize to higher $d$.

\begin{proof} We use the simple closed form expression for $s_2(n+1)$, namely that it equals $\ncr{2n}{n}$. By the definition of the second standardized moment and use of \eqref{simpledDimLatticeMean} where $d \ = \ 2$, we have
\begin{eqnarray}
\sigma_2(n+1)^2 \ = \ \frac{\sum^{n + 1}_{k = 1}k^2\ncr{n}{k - 1}^2}{\ncr{2n}{n}} - \left(\frac{1}{2}n + 1\right)^2.
\label{simpleSquareLatticeStdDevExpA}
\end{eqnarray}	
Shifting the index of summation to start at $k = 0$ and expanding yields
\begin{eqnarray}
\sigma_2(n+1)^2 & \ = \  &\frac{\sum^{n}_{k = 0}(k + 1)^2\ncr{n}{k}^2}{\ncr{2n}{n}} - \left(\frac{1}{2}n + 1\right)^2 \nonumber\\
& \ = \ & \frac{\sum^{n}_{k = 0}k^2\ncr{n}{k}^2}{\ncr{2n}{n}} + \frac{\sum^{n}_{k = 0}2k\ncr{n}{k}^2}{\ncr{2n}{n}} + \frac{\sum^{n}_{k = 0}\ncr{n}{k}^2}{\ncr{2n}{n}} - \left(\frac{1}{2}n + 1\right)^2  \nonumber\\
& \ = \ & \frac{\sum^{n}_{k = 0}k^2\ncr{n}{k}^2}{\ncr{2n}{n}} + 2\mu_2(n + 1) - \frac{\sum^{n}_{k = 0}\ncr{n}{k}^2}{\ncr{2n}{n}} - \left(\frac{1}{2}n + 1\right)^2.
\label{simpleSquareLatticeStdDevExpAA}
\end{eqnarray}
Using \eqref{simpleSquareLatticeMean} for the mean and recalling that $\sum^{n}_{k = 0}\ncr{n}{k}^2 = \ncr{2n}{n}$, we have
\begin{eqnarray}
\sigma_2(n+1)^2 & \ = \ & \frac{\sum^{n}_{k = 0}k^2\ncr{n}{k}^2}{\ncr{2n}{n}} + 2\left(\frac{1}{2}n + 1\right) - 1 - \left(\frac{1}{2}n + 1\right)^2 \nonumber\\ & \ = \ & \frac{\sum^{n}_{k = 0}k^2\ncr{n}{k}^2}{\ncr{2n}{n}} - \frac{n^2}{4}.
\label{simpleSquareLatticeStdDevExpAB}
\end{eqnarray}
We now use the identity
\begin{eqnarray}
\sum^{n}_{k = 0}k^2\ncr{n}{k}^2 \ = \ n^2\ncr{2n - 2}{n - 1},\label{secondOrderSquareCoeff}
\end{eqnarray}
which we quickly prove for completeness. To see this, expand the binomial coefficient and cancel $k$'s:
\begin{eqnarray}
\sum^{n}_{k = 0}k^2\ncr{n}{k}^2 \ = \ \sum^{n}_{k = 0}k^2\left(\frac{n!}{k!(n - k)!}\right)^2 \ = \ \sum^{n}_{k = 1}n^2\left(\frac{(n - 1)!}{(k - 1)!(n - k)!}\right)^2.
\label{secondOrderSquareCoeffProofA}
\end{eqnarray}
Shifting indices, we can rewrite the above as \begin{eqnarray} \sum^{n}_{k = 0}k^2\ncr{n}{k}^2 \ = \ n^2 \sum_{\ell =0}^{n-1} \ncr{n-1}{\ell}^2 \ = \ n^2 \sum_{\ell=0}^{n-1} \ncr{n-1}{\ell} \ncr{n-1}{n-1-\ell}, \end{eqnarray} and as we have seen numerous times the sum equals $\ncr{2n-2}{n-1}$ (it is the number of ways to choose $n-1$ objects from $2n-2$, where we consider $n-1$ of the items to be in one set and the remaining $n-1$ in another).
Substituting \eqref{secondOrderSquareCoeff} into \eqref{simpleSquareLatticeStdDevExpAB} gives
\begin{eqnarray}
\sigma_2(n+1)^2 & \ = \ & \frac{n^2\ncr{2n - 2}{n - 1}}{\ncr{2n}{n}} - \frac{n^2}{4} \nonumber\\
& \ = \ &\frac{n^3}{4n - 2} - \frac{n^2}{4}  \ = \ \frac{n^2}{8n - 4}.
\label{simpleSquareLatticeStdDevExpB}
\end{eqnarray}	
Taking the square root of both sides of \eqref{simpleSquareLatticeStdDevExpB} gives the desired result.
\end{proof}

We remark on the difficulty in generalizing the above argument to arbitrary $d$. The problem is in \eqref{secondOrderSquareCoeff}. There it was crucial that $d \ = \ 2$, as we then canceled the $k^2$ with the two factors of $k$ in the denominator. In higher dimensions we do not have such perfect alignment.


\ \\
\end{document}